\documentclass[ECP]{ejpecp} 

\SHORTTITLE{Positivity of hit-and-run and related algorithms} 

\TITLE{Positivity of hit-and-run and related algorithms\thanks{
The first author was supported by the DFG Priority Program 1324 and the DFG
Research Training Group 1523.
The second author was supported by the ERC Advanced Grant 228032
PTRELSS and the DFG Research Training Group 1523.
}} 

\AUTHORS{%
  Daniel~Rudolf\footnote{Universit\"at Jena, Germany.
    \EMAIL{daniel.rudolf@uni-jena.de}}
  \and 
  Mario~Ullrich\footnote{Universit\`a Roma Tre, Italy.
    \EMAIL{ullrich.mario@gmail.com}}}


\AMSSUBJ{60J05} 

\SUBMITTED{December 17, 2012} 
\ACCEPTED{June 23, 2013} 

\VOLUME{0}
\YEAR{2012}
\PAPERNUM{0}
\DOI{vVOL-PID}

\ABSTRACT{We prove positivity of the Markov operators that correspond to 
the hit-and-run algorithm, 
random scan Gibbs sampler, slice sampler and 
Metropolis algorithm with positive proposal. 
In particular, the results show that it is not necessary 
to consider the lazy versions of these Markov chains.
The proof relies on a well known lemma which relates 
the positivity of the product $M T M^*$, for some operators 
$M$ and $T$, to the positivity of $T$. 
It remains to find that kind of representation of the 
Markov operator with a positive operator $T$.}

\newcommand{\vol}{{\rm vol}}

\newcommand{\B}{\mathcal{B}}  
\newcommand{\F}{\mathcal{F}}

\newcommand{\N}{\mathbb{N}}
\newcommand{\R}{\mathbb{R}}
\newcommand{\abs}[1]{\left\vert #1 \right\vert}	
\newcommand{\scalar}[2]{\langle #1 , #2 \rangle}
\newcommand{\set}[1]{\lbrace #1 \rbrace}
\renewcommand{\P}{\mathbb{P}}

\renewcommand{\a}{\alpha}

\renewcommand{\O}{\Omega}
\newcommand{\dint}{\text{\rm d}}

\newcommand{\gap}{{\rm {gap}}}
\newcommand{\spec}{{\rm {spec}}}

\begin{document}

\section{Introduction}

In many applications, for example volume computation 
\cite{KaLoSi97,LoSi93,LoVe06-2} or integration of functions 
\cite{LoVe06-1,MaNo07,Ru09,Ru12}, 
it is essential that one can
approximately sample a distribution in a convex body. 
The dimension $d$ might be very large.
One approach that is feasible for a general 
class of problems 
is to run a Markov chain 
that has the desired 
distribution as its limit distribution. 
In the following let us explain 
why the positivity of
the Markov operator is helpful to prove efficiency results 
for such sampling procedures.

We assume that we have a Markov chain in $K\subset \R^d$ which is reversible 
with respect to (w.r.t.) the distribution $\pi$.
Let $P\colon L_2(\pi)\to L_2(\pi)$ be the corresponding Markov operator and let
$L_2(\pi)$ be all (w.r.t. $\pi$) square integrable functions 
$f\colon K \to \R$. 
We assume that $P$ is \emph{ergodic}, which means that 
$Pf=f$ implies that $f$ is constant. 
Then let $\gap(P)=1-\beta$ be the \emph{absolute spectral gap}, 
where $\beta$ denotes 
the largest absolute value of the elements of the spectrum of $P$ without $1$.
In formulas $\beta = \sup \set{ \abs{\a}\colon \a \in \spec(P)\setminus 1 }$,
where $\spec(P)$ denotes the spectrum of $P$.
For example a lower bound for $\gap(P)$ implies an upper bound 
of the total variation distance \cite{LoSi93}
and 
on the mean square error of Markov chain Monte Carlo algorithms
for the approximation of expectations with respect to $\pi$, 
 see e.g.~\cite{Ru12}.

Maybe the most successful technique to bound $\gap(P)$ is the 
conductance technique \cite{LaSo88,LoSi93}. 
But, unfortunately, 
bounds on the conductance allow only bounds on the 
second largest element of the spectrum of the Markov operator.
This is known as Cheeger's inequality \cite{LaSo88}.
To handle variation distance and absolute spectral gap it is necessary 
to consider also the smallest element of the spectrum, 
which describes some kind of periodicity of Markov chains.
Usually, this problem is avoided by considering the \emph{lazy version}
of a Markov chain. That is, in each step, the Markov 
chain remains at the current state with probability $1/2$.
Such a lazy version induces a Markov operator with non-negative spectrum, which 
implies that the smallest element of the spectrum does not matter.
This strategy has almost no influence on the computational 
cost, since, compared to the overall cost of one step of the chain, 
one additional random number is mostly negligible. 
However, it is desirable to omit any slowdown whenever possible.

In particular, the best known bounds on the total variation distance
of the hit-and-run algorithm, see \cite{Lo99,LoVe06-1,LoVe06,LoVe07},
rely on the conductance and Corollary~1.5 resp.~Corollary~1.6 of \cite{LoSi93}.
These corollaries give upper bounds on the total variation distance 
in terms of the conductance resp.~$s$-conductance, but 
it has to be assumed that the corresponding Markov operator is positive, 
cf.~Section~1.B of \cite{LoSi93}.
(More precisely, the assumption that the smallest element of the spectrum 
is smaller in absolute value than the second largest one is sufficient.)
Thus, there is a small gap in the proofs of \cite{LoVe06-1} and \cite{LoVe06},
which might be easily fixed by considering the lazy version of the hit-and-run algorithm.
We prove, among others, that hit-and-run is positive. 
Thereby, we close the small gap and show in addition that 
the results of \cite{Lo99} and \cite{LoVe07} hold also for the 
non-lazy hit-and-run algorithm 
as originally proposed in \cite{Sm84}.

The technique that we will use to prove that the spectrum 
of a Markov operator is positive is based on a simple and 
well known lemma from functional analysis. 
This was already successfully applied in a discrete setting 
to prove positivity (and comparison results) for the 
Swendsen-Wang process from statistical physics, see~\cite{Ul12b,U-phd}.
Here, we show that the hit-and-run algorithm, 
random scan Gibbs sampler, slice sampler and the 
Metropolis algorithm with positive proposal are positive. 
In particular, it implies that the independent Metropolis algorithm is positive.
The result is new for the hit-and-run algorithm and the Metropolis algorithm 
with positive proposal, whereas for the random scan Gibbs sampler and the 
slice sampler it is known \cite{LiWoKo95,MiTi02}.

\section{The procedure} \label{sec:prelim}

We consider a time-homogeneous
Markov chain $(X_i)_{i\in\N}$, 
where the $X_i$ 
are random variables on a 
common probability space $(\O,\F,\P)$ that map 
into $\R^d$, equipped with the Borel $\sigma$-algebra $\B$,  
and satisfy the Markov property. 
Namely,
\[
\P(X_n\in A_n \mid X_{n-1}\in A_{n-1},...,X_0\in A_0)
\,=\, \P(X_n\in A_n \mid X_{n-1}\in A_{n-1})
\]
for all $n\ge1$ and any sequence of 
$\B$-measurable sets $A_0,A_1,\dots$ 
with the property
$\P(X_{n-1}\in A_{n-1},...,X_0\in A_0)>0$. 
We assume that the Markov chain has a unique stationary 
distribution $\pi$ and that it is reversible with respect to 
this measure. 
For a more comprehensive  introduction to the theory of 
Markov chains we refer to \cite{MeTw09,RoRo04}.

To every Markov chain $(X_i)_{i\in\N}$ corresponds a 
\emph{Markov kernel} 
$P\colon\R^d\times\B\to[0,1]$ 
such that for each $x\in\R^d$, 
$P(x,\cdot)$ is a probability measure on 
$\B$ and, for each $A\in\B$, $P(\cdot,A)$ is $\B$-measurable.
This Markov kernel is 
 given by
\[
P(x,A) \,=\, \P(X_{n+1}\in A \mid X_{n}=x),\quad x\in \R^d,\;A\in\B,\:n\in\N
\]
and describes the probability that the Markov chain reaches 
the set $A$ in one step 
from $x$.
Using this Markov kernel we define the \emph{Markov operator} 
$P$ 
(for notational convenience we use the same letter as for the 
Markov kernel) 
by 
\[
Pf(x) \,=\, \int_{\R^d} f(y) \,P(x,\dint y),\quad x\in\R^d
\]
for all functions $f\in L_2=L_2(\pi)$, where $L_2$ is the 
Hilbert space of functions on $f\colon \R^d \to \R$ with inner product
\[
\scalar{f}{g} \,=\, \int_{\R^d} f(x)\, g(x) \,\dint\pi(x).
\]
By reversibility of the Markov chain we know that $P$ 
is a self-adjoint operator on $L_2$.
A self-adjoint operator $P$ is called \emph{positive}, 
written $P\ge0$, if 
\[
\scalar{Pf}{f} \,\ge\, 0, \quad \forall f\in L_2.
\]
It is well known that
positive operators have only non-negative 
spectrum, for further details see for example \cite{Kr89}.

Our aim is to show that several Markov chains that are used 
to sample from distributions in $\R^d$ induce positive 
Markov operators. In this case, we say that the 
Markov chain is positive. 
We will basically utilize the following lemma.
\begin{lemma} \label{lem: pos}
 Let $H_1$ and $H_2$ be Hilbert spaces and $M\colon H_1 \to H_2$ 
 be a bounded, linear operator. Let $M^*$ be the adjoint operator of $M$ 
 and
 let $T\colon H_2 \to H_2 $ be a bounded, linear and positive operator. 
 Then $M T M^* \colon H_1 \to H_1$ is also positive.
\end{lemma}
\begin{proof}
 We denote the inner product of $H_i$ by $\scalar{\cdot}{\cdot}_i$ for $i=1,2$. 
 By the definition of the adjoint operator and positivity of $T$,
  \[
    \scalar{M T M^* f}{f}_1 = \scalar{TM^* f}{M^* f}_2 \geq 0.
  \] 
 This proves the statement.
\end{proof}

Suppose we have an operator $P\colon H_1 \to H_1$ on a Hilbert space 
with the property that it can be written as $P=MTM^*$, 
where  $T\colon H_2 \to H_2$, $M\colon H_1 \to H_2$ and 
$M^*$ is the adjoint of $M$ for some (other) Hilbert space $H_2$.
If we can show, additionally, that $T$ is a positive operator, 
we obtain by the lemma above that $P$ is also positive.
Thus, the proof of positivity of the Markov chains under 
consideration is done by a construction of a suitable second 
Hilbert space such that the corresponding Markov operator can 
be written in the above mentioned form.

\section{Applications}

Throughout this section we consider Markov chains 
in a subset $K$ of $\R^d$ with non-empty interior.
Additionally, we denote by $B_d$ the $d$-dimensional unit ball 
and by $S^{d-1}$ its boundary.
Let $\rho \colon K \to [0,\infty)$ be 
a (not necessarily normalized) density, i.e.~a non-negative 
Lebesgue-integrable function. 
We define the measure with density $\rho$ by
\[
\pi(A)=\frac{\int_A \rho(x)\;\dint x}{ \int_K \rho(x)\;\dint x}
\]
for all measurable sets $A\subset K$.
For example, if $\rho(x)=\mathbf{1}_K(x)$ then $\pi$ is simply 
the uniform distribution on $K$. 
In what follows we present some Markov chains that can be used 
sample approximately from $\pi$, that is, $\pi$ is their stationary 
distribution. 
We will see that each of them is positive, independent from the 
choice of the density $\rho$.

We will define only the Markov operators 
for the corresponding Markov chains, since 
the corresponding Markov kernel can be obtained by 
applying the operators to indicator functions.

\subsection{Hit-and-run}
\label{subsec: har}

The hit-and-run algorithm 
consists of two steps: Starting from $x\in K$, choose 
a random direction $\theta\in S^{d-1}$ and then choose the next 
state of the Markov chain with respect to the density $\rho$ 
restricted to the chord determined by $x$ and $\theta$.

For $x\in K$ and $\theta\in S^{d-1}$ 
we denote by $L(x,\theta)$ the \emph{chord} in $K$ 
through $x$ and $x+\theta$, i.e.
\[
  L(x,\theta) = \set{x+s\theta \in K \mid s\in \R}.
\]
Additionally we write $\kappa_d$ for the volume of the 
$(d-1)$-dimensional unit sphere and 
\begin{equation} \label{eq:l}
 \ell(x,\theta) = \int_{L(x,\theta)}  \rho(y)\, \dint y
\end{equation}
for the total weigth of the chord $L(x,\theta)$.
The Markov operator $H$ that corresponds to the hit-and-run chain 
is defined by 
\[
Hf(x) 
\,=\, \frac{1}{\kappa_d} \int_{S^{d-1}} 
	\frac{1}{\ell(x,\theta)} \int_{L(x,\theta)}
	f(y)\, \rho(y)\; \dint y\, \dint\theta
\]
for all $f\in L_2(\pi)$.
To rewrite $H$ in the desired form 
let $\mu$ be the product measure of $\pi$ and the uniform 
distribution on $S^{d-1}$ 
and $L_2(\mu)$ be the Hilbert space of 
functions $g:K\times S^{d-1}\to\R$ with inner-product
\[
  \scalar{g_1}{g_2}_\mu =\frac{1}{\kappa_d} \int_K \int_{S^{d-1}} 
  g_1(x,\theta)\, g_2(x,\theta)\; \dint \theta\,\dint\pi(x)
\]
for $g_1,g_2\in L_2(\mu)$.
We define the operators $M\colon L_2(\mu)\to L_2(\pi)$ 
and $T\colon L_2(\mu) \to L_2(\mu)$ by
\[
Mg(x) \,=\, \frac{1}{\kappa_d} \int_{S^{d-1}} 
	g(x,\theta)\, \dint\theta
\]
and
\[
Tg(x,\theta) \,=\, \frac{1}{\ell(x,\theta)}\int_{L(x,\theta)}
	g(y,\theta)\, \rho(y)\, \dint y.
\]
Recall that the adjoint operator of $M$ is the unique operator $M^*$ 
that satisfies $\scalar{f}{Mg}=\scalar{M^*f}{g}_\mu$ for all 
$f\in L_2(\pi)$, $g\in L_2(\mu)$, see~\cite[Thm.~3.9-2]{Kr89}.
Since
\[
\scalar{f}{Mg} \,=\, \frac{1}{\kappa_d} \int_K \int_{S^{d-1}} 
	f(x)\, g(x,\theta)\, \dint\theta\, \dint\pi(x),
\]
we obtain that, for all $\theta\in S^{d-1}$ and $x\in K$,
\[
M^*f(x,\theta) \,=\, f(x).
\]
This implies 
\[
MTM^*f(x) \,=\, \frac{1}{\kappa_d} \int_{S^{d-1}} 
	\frac{1}{\ell(x,\theta)}\int_{L(x,\theta)} 
	f(y)\,\rho(y)\; \dint y\, \dint\theta 
\,=\, Hf(x)
\]
and thus, that $M$ and $T$ are the desired ``building blocks'' 
for Lemma~\ref{lem: pos}.
First of all, note that by Fubini's Theorem the operator 
$T$ is self-adjoint in $L_2(\mu)$.
It remains to show that $T$ is positive.
We know that $L(x,\theta)=L(y,\theta)$ for all 
$y\in L(x,\theta)$. 
It follows that 
\[\begin{split}
T^2g(x,\theta) 
\,&=\, \frac{1}{\ell(x,\theta)}\int_{L(x,\theta)}
	Tg(y,\theta)\,\rho(y)\; \dint y \\
\,&=\, \frac{1}{\ell(x,\theta)}\int_{L(x,\theta)}
	\frac{1}{\ell(y,\theta)}\int_{L(y,\theta)}
	g(z,\theta)\,\rho(z)\, \dint z\;\rho(y)\, \dint y \\
&=\, \frac{1}{\ell(x,\theta)^2}
	\int_{L(x,\theta)}
	g(z,\theta)\,\rho(z)\,\dint z\;\int_{L(x,\theta)}\rho(y)\,\dint y\\
&=\, Tg(x,\theta).
\end{split}\]
Thus, $T$ is a self-adjoint and idempotent operator on 
$L_2(\mu)$, which implies that $T$ is a projection and, 
in particular, that it is positive, see e.g.~\cite[Thm.~9.5-2]{Kr89}.
Finally, Lemma~\ref{lem: pos} shows that $H$ is positive.

\subsection{Gibbs sampler}
The Gibbs sampler, or specifically the random scan Gibbs sampler, 
is conceptually very similar to the hit-and-run algorithm. 
In each step, we choose a direction and sample
with respect to $\rho$ restricted to the chord in this direction.
But, in contrast to the hit-and-run, 
we choose the direction from the $d$ possible directions 
of the coordinate axes.

Let $e_1=(1, 0, \dots ,0), \dots , e_d=(0,\dots,0,1)$ be 
the Euclidean standard basis in $\R^d$ 
and $\ell(\cdot,\cdot)$ be from \eqref{eq:l}.
The Markov operator $G$ of the Gibbs sampler is given by
\[
  Gf(x) = \frac{1}{d} \,\sum_{j=1}^d \frac{1}{\ell(x,e_j)} 
  	\int_{L(x,e_j)} f(y)\,\rho(y)\; \dint y, 
\]
for all $f\in L_2(\pi)$.
We follow almost the same lines as for the hit-and-run chain. 
Let $m$ be the product measure of $\pi$ and the uniform 
distribution on $[d]=\set{1,\dots,d}$. 
By $L_2(m)$ 
we denote the Hilbert space of functions 
$g\colon K \times [d] \to \R$ equipped with the inner product
\[
  \scalar{g_1}{g_2}_m = \frac{1}{d} \sum_{j=1}^d 
  	\int_K g_1(x,j)\, g_2(x,j)\, \dint \pi(x)
\]
for $g_1,g_2\in L_2(m)$.
We define the operators $M\colon L_2(m) \to L_2(\pi)$ and 
$T\colon L_2(m)\to L_2(m)$
by
\[
  Mg(x) = \frac{1}{d} \sum_{j=1}^d g(x,j)
\]
and 
\[
  Tg(x,j)=\frac{1}{\ell(x,e_j)} 
  	\int_{L(x,e_j)} g(y,j)\, \rho(y)\, \dint y.
\]
By the same calculations as in Subsection~\ref{subsec: har} 
we obtain for all $f\in L_2(\pi)$, $x\in K$ and $j\in[d]$, that
$M^*f(x,j) = f(x)$  
and that $G=MTM^*$. 
It is easily seen that $T$ is self-adjoint and idempotent.
Hence, $T$ is a projection and thus, positive, which 
proves the assertion by Lemma~\ref{lem: pos}.

\subsection{Slice sampler}
\label{subsec: gss}
For any $t > 0$ assume that $R_t$ is the transition kernel of a Markov chain on the 
level set $K(t)$ of $\rho$, i.e.
\[
K(t)=\set{x\in K\mid \rho(x) \geq t}.
\]
Note that $\vol_d(K(t))< \infty$ for each $t> 0$ by the integrability of $\rho$. 
Further let $R_t$ be reversible with respect to $U_t$, the uniform distribution on $K(t)$, i.e.
\[
U_t(A) = \frac{\vol_d(A\cap K(t))}{\vol_d(K(t))}, \quad A\subset K,\;t>0,
\]
where $\vol_d$ denotes the $d$-dimensional Lebesgue measure.
Note that if $K=K(0)$ is bounded, then also $U_0$ is well-defined and denotes the uniform
distribution in $K$. 
The slice sampler, starting from a state $x$ works as follows: 
First choose a level $t$ uniformly distributed in $(0,\rho(x)]$ and then sample 
the next state with respect to $R_t(x,\cdot)$ in the level set $K(t)$.
If $R_t(x,\cdot) = U_t(\cdot)$ then the slice sampler is called simple slice sampler \cite{MiTi02}.
The corresponding Markov operator is defined by
\[
R f(x) \,=\, \frac{1}{\rho(x)} \int_0^{\rho(x)} R_t f(y)\, \dint t
\,=\, \frac{1}{\rho(x)} \int_0^{\rho(x)} \int_{K(t)} f(y)\, R_t(x,\dint y)\, \dint t,
\]
for all $f\in L_2(\pi)$.
For any $t> 0$ we assume that $R_t$ is a positive operator on $L_2(U_t)$, 
which is the set of all square integrable real functions with respect $U_t$ 
on $K(t)$, i.e. 
\[
  \scalar{R_t f}{f}_{U_t} 
  \,=\, \int_{K(t)} R_t f(x)\,f(x)\; U_t(\dint x) \,\geq\, 0.
\]
To show that $R$ is positive, let 
\[
 K_\rho = \set{ (x,t)\mid x\in K,\; t\in(0,\rho(x)] } \subset \R^{d+1}
\]
and let $\mu$ be the uniform distribution in $K_\rho$. 
Let $L_{2}(\mu)$ be the Hilbert space of functions $g\colon K_\rho \to \R$ 
with inner product 
\[
  \scalar{g_1}{g_2}_\mu = \int_{K_\rho} g_1(x,t)\, g_2(x,t)\;\dint\mu(x,t)
\]
for $g_1,g_2 \in L_2(\mu)$.
We define the operators $M \colon L_{2}(\mu) \to L_{2}(\pi)$ 
and $T \colon L_{2}(\mu) \to L_{2}(\mu)$ by
\[
Mg(x) = \frac{1}{\rho(x)} \int_0^{\rho(x)} g(x,t)\;\dint t.
\]
and
\[
  Tg(x,t) = \int_{K(t)} g(y,t)\; R_t(x,\dint y)
\]
for $g\in L_2(\mu)$. The adjoint operator $M^* \colon L_{2}(\pi) \to L_{2}(\mu)$ is
$M^*f(x,t) = f(x)$.
The operator $T$ is self-adjoint, since $R_t$ is reversible with respect to $U_t$.
For the positivity define $g_t(x)=g(x,t)$, $(x,t)\in K_\rho$. 
We have
\begin{align*}
  \scalar{Tf}{f}_\mu \,=\, \int_0^\infty \scalar{R_t g_t}{g_t)}_{U_t}\; 
	\frac{\vol_d(K(t))}{\vol_{d+1}(K_\rho)} \,\dint t,
\end{align*}
which implies positivity of $T$ by the positivity of $R_t$.
By the fact that $R=M T M^*$ and by Lemma~\ref{lem: pos} it is proven that $R$ is positive. 

\subsection{Metropolis algorithm}
For simplicity we additionally assume that $K\subset \R^d$ is bounded. With
some extra work one could avoid this assumption.
Let $B$ be a positive proposal kernel which is reversible 
with respect to $U_0$, the uniform distribution in $K$.
Then the Markov operator of the Metropolis algorithm is given by
\[
 Mf(x)= \int_K f(y) \a(x,y) B(x,\dint y) + \left(1-\int_K \a(x,y) B(x,\dint y)\right) f(x)
\]
where $\a(x,y)=1 \wedge \frac{\rho(y)}{\rho(x)}$ and $f\in L_2(\pi)$.
We interpret the Metropolis algorithm as a slice sampler.
For $t\geq0 $, $x\in K(t)$ and $A\subset K$
define 
\[
 R_t(x,A)= B\bigl(x,A\cap K(t)\bigr) + 
	\left(1-B\bigl(x,K(t)\bigr) \right) \mathbf{1}_A(x).
\]
Recall that 
\[
R f(x) = \frac{1}{\rho(x)} \int_0^{\rho(x)} \int_{K(t)} f(y)\; R_t(x,\dint y)\; \dint t,
\]
is the Markov operator of the slice sampler and that 
$U_t$ is the uniform distribution in $K(t)$.

\goodbreak

\begin{lemma}
\begin{enumerate}
  \item\label{en: self_adj}
  If $B$ is reversible with respect to $U_0$, then $R_t$ is reversible with respect to $U_t$ for any $t\geq0$.
 \item\label{en: pos}
  If $B$ is positive on $L_2(U_0)$, then $R_t$ is positive on $L_2(U_t)$ for any $t\geq0$.
 \item\label{en: metro_slice} 
  The general slice sampler and the Metropolis algorithm coincide, i.e. $Rf = Mf$ for $f\in L_2(\pi)$.
\end{enumerate}
\end{lemma}
\begin{proof} 
We have for any $f,g \in L_2(U_t)$ that
\begin{align*}
 \scalar{R_t f}{g}_{U_t} 
& = \int_{K(t)} \int_{K(t)} f(y) g(x) \mathbf{1}_{K(t)}(y)\,B(x,\dint y)\; U_t(\dint x)\\
&\qquad + \int_{K(t)} (1-B(x,K(t))) f(x)g(x)\; U_t({\dint x})\\
& = \scalar{B(\mathbf{1}_{K(t)}f)}{\mathbf{1}_{K(t)}g}_{U_0} \, \frac{\vol_d(K)}{\vol_d(K(t))}\\
&\qquad + \int_{K(t)} (1-B(x,K(t))) f(x)g(x)\; U_t({\dint x}).
\end{align*}  
By using the self-adjointness of $B$ on $L_2(U_0)$ we obtain that $R_t$ is self-adjoint on $L_2(U_t)$
for any $t\geq0$, which proves \eqref{en: self_adj}. Assertion \eqref{en: pos} follows by similar
arguments. One obtains
\begin{align*}
 \scalar{R_t f}{f}_{U_t} 
& = \scalar{B(\mathbf{1}_{K(t)}f)}{\mathbf{1}_{K(t)}f}_{U_0} \, \frac{\vol_d(K)}{\vol_d(K(t))}\\
&\qquad + \int_{K(t)} (1-B(x,K(t))) f(x)^2\; U_t({\dint x}),
\end{align*}
which, by using the positivity of $B$, proves the positivity of $R_t$. 
Note that the Markov operator 
of the slice sampler can be written as
\begin{align*}
 Rf(x) & = \frac{1}{\rho(x)} \int_0^{\rho(x)} \int_{K(t)} f(y)\, B(x,\dint y)\, \dint t \\
& \qquad + \frac{f(x)}{\rho(x)} \int_0^{\rho(x)} \mathbf{1}_{K(t)}(x)\, (1-B(x,K(t)))\, \dint t \\
& = \frac{1}{\rho(x)}\int_{K } 
   \int_0^{\infty}  \mathbf{1}_{K(t)}(x) \mathbf{1}_{K(t)}(y)\,\dint t\,  f(y)\, B(x,\dint y)  \\
& \qquad + f(x) 
\left(1-\frac{1}{\rho(x)} \int_K \int_0^{\infty} \mathbf{1}_{K(t)}(x) 
\mathbf{1}_{K(t)}(y)\, \dint t \, B(x,\dint y) \right).
\end{align*}
Then \eqref{en: metro_slice} follows by 
\[
\frac{1}{\rho(x)} \int_0^\infty  \mathbf{1}_{K(t)}(x) \mathbf{1}_{K(t)}(y)\, \dint t \,=\, \a(x,y).
 \]

\end{proof}

Note that, by the previous lemma all assumptions of Subsection~\ref{subsec: gss} 
are satisfied if we additionally assume that $B$ is positive.
Hence, the Metropolis algorithm defines 
a positive Markov operator if the proposal is positive.




\providecommand{\bysame}{\leavevmode\hbox to3em{\hrulefill}\thinspace}
\providecommand{\MR}{\relax\ifhmode\unskip\space\fi MR }
\providecommand{\MRhref}[2]{%
  \href{http://www.ams.org/mathscinet-getitem?mr=#1}{#2}
}
\providecommand{\href}[2]{#2}


\ACKNO{The second author wants to thank F. Martinelli and 
P. Caputo for their kind hospitality while he was
visiting Universit\'a Roma Tre, Italy. 
Additionally, we want to thank K. \L atuszy\'{n}ski and the 
referees for helpful comments.}


\end{document}